\documentclass[12pt,a4paper]{article}
\usepackage{amsmath,amssymb}
\usepackage[mathscr]{euscript}
\usepackage{graphicx,xr}
\usepackage{amsfonts,amsthm}
\usepackage{latexsym}
\usepackage[T1]{fontenc}
\usepackage[utf8]{inputenc}
\usepackage{pxfonts}

\usepackage{color}
\usepackage{tikz}
\usetikzlibrary{automata,arrows,positioning,calc}

\usepackage{multirow}

\usepackage[left=2cm,right=2cm,top=2cm,bottom=2cm]{geometry}

\newtheorem{teo}{Theorem}
\newtheorem*{teob}{Theorem}
\newtheorem{lema}[teo]{Lemma}

\newcommand{\A}{{\mathcal A}}
 
\newcommand{\R}{\mathbb R}
\newcommand{\rr}{\overline{\R}^+}
\newcommand{\N}{\mathbb N}

\def\keywords{\vspace{.5em}
{\textit{Keywords}:\,\relax%
}}

\title{Study of the transcendence of a family of generalized 
continued fractions}
\author{Túlio O. Carvalho} 
\date{April 14, 2020}
\vspace{2cm} 

\begin{document}

\maketitle

\begin{abstract}
We study a family of generalized continued fractions, which are 
defined by a pair of substitution sequences in a finite 
alphabet. We prove that they are {\em stammering} sequences, in 
the sense of Adamczewski and Bugeaud. We also prove that this 
family consists of transcendental numbers which are not 
Liouvillian.
We explore the partial quotients of their regular continued 
fraction expansions, arriving at no conclusion concerning their 
boundedness. 
\end{abstract}


\section{Introduction}

The problem of characterizing continued fractions of numbers 
beyond rational and quadratic has received consistent attention 
over the years. 
One direction points to an attempt to understand properties of 
algebraic numbers of degree at least three, but at times even 
this line ends up in the realm of transcendental numbers. 

Some investigations \cite{bruno,garrity,schwe} on algebraic 
numbers depart from generalizations of continued fractions. This 
line of investigation has been tried since Euler, see 
\cite{schwe,bruno} and references therein, with a view to 
generalize Lagrange's theorem on quadratic numbers,  in search 
of proving a relationship between algebraic numbers and 
periodicity of multidimensional maps yielding a
sequence of approximations to a irrational number. 
This theory has been further developed for instance  to the 
study of ergodicity of the triangle map \cite{meno}. In fact, a 
considerable variety of algorithms may be called 
generalizations of continued fractions: for instance  
\cite{keane}, Jacobi-Perron's and Poincaré's algorithms in 
\cite{schwe}. 

We report on a study of generalized continued fractions of 
the form:
\begin{equation}
\label{thab}
\theta(a,b)\stackrel{{\rm def}}{=} 
a_0+\frac{b_0}{a_1+\frac{b_1}{a_2+\frac{b_2}{a_3+\frac{b_3}{
a_4+\frac { b_4 } { \ddots}}}}}
 \ , 
\end{equation}
(where $a_0\geq 0$, $a_n\in \N$, for $n\in \N$ and $b_n \in 
\N$, for $n\geq 0$), investigating a class with a 
regularity close, in a 
sense, to periodicity. This family of generalized continued 
fractions converges when $(a_n)$ and $(b_n)$ are 
finite valued sequences. They were considered {\em formally} in 
\cite{list}, in the context exemplified in Section \ref{exam1}. 
The {\em stammering} sequences 
\cite{acta}, and sequences generated by morphisms 
\cite{abd,aldavquef,allsha}, consist in a natural step away 
from 
periodic ones. Similarly to the results in 
\cite{abd,aldavquef}, on regular continued fractions, we prove 
that the family of numbers considered are transcendental.

\begin{teo}
\label{teo1}
Suppose $(a_n)$ and $(b_n)$, $n\geq 0$, are fixed points of 
primitive substitutions, then the number $\theta(a,b)$ is 
transcendental.  
\end{teo}

\def\S{\mathsf S}
We also consider Mahler's classification within this family of 
transcendental numbers, proving that they cannot be 
Liouvillian. 
Let us call the numbers of the 
family of generalized 
continued fractions $\theta(a,b)$, with ${\mathbf a}$ and 
${\mathbf b}$ stammering sequences coming from a primitive 
substitution number of {\em type} $\S_3$.

\begin{teo}
Type $\S_3$ numbers are either $S$-numbers or 
$T$-numbers in Mahler's classification.
\label{teo2}
\end{teo}

The paper is organized as follows. In Section \ref{conv}, we 
prove the convergence of the generalized continued fraction 
expansions for type $\S_3$ numbers. In Section \ref{transc}, we 
prove the transcendence of type $\S_3$ numbers. In 
Section \ref{liou}, we use Baker's Theorem to prove that they 
are either $S$-numbers or $T$-numbers. In Section \ref{exam1}, 
we show some inconclusive calculations on the partial quotients 
of the regular continued fraction of an specific type $\S_3$ 
number.

\section{Convergence}
\label{conv}

We start from the analytic
theory of continued fractions \cite{wall} to prove the 
convergence of \eqref{thab} when $(a_n)$ and $(b_n)$, $n\geq 
0$, are sequences in a finite alphabet.

Let $\rr=[0,\infty]$ denote the extended positive real 
axis, 
with the understanding that $a+\infty=\infty$, for any 
$a\in \rr$; $a\cdot \infty=\infty$, if $a>0$, 
$0\cdot \infty = 0$ and $a/\infty=0$, if $a\in \R$. We do not 
need to define $\infty/\infty$. 

Given the sequences $(a_k)$ and $(b_k)$ of non-negative 
(positive) integers, consider the Möbius transforms 
$t_k:\rr\to \rr$
\[ t_k(w)= a_k+\frac{b_k}{w} \ ,\; k\in \N \ ,\]
and their compositions 
\[ t_it_j(w)= t_i(a_j+b_j/w)=a_i+b_i/(a_j+b_j/w) \ .\]
This set of Möbius transforms is closed under compositions and 
form a 
semigroup. 

It is useful to consider the natural correspondence between Möbius 
transformations and $2\times 2$ matrices: 
\[  
M_k=\begin{pmatrix} a_k & b_k \\ 1 & 0 \end{pmatrix} \ .\]
Taking the 
positive real cone ${\cal C}_2=\{(x,y) \ |\ 
x\geq 0\ , \ y\geq 0\ , x+y>0\}$ with the equivalence $(x,y) \sim 
\lambda (x,y)$ for every $\lambda>0$, we 
have an homomorfism between the semigroup of Möbius transforms, 
under composition, acting on $\rr$ and the algebra of matrices above 
(which are all invertible) acting on ${\cal C}_2/\sim$. 

Assume the limit
\[ \lim_{n\to \infty} t_0t_1t_2\cdots t_n(0)  \]
exists as a real positive number, then it is given once we 
know the sequences $(a_n)$ 
and $(b_n)$, $n\geq 0$. 
In this case, it is equal to $\lim_{n\to \infty} 
t_0t_1t_2\cdots t_{n-1}(\infty)$ as well, so that the initial 
point may 
be taken as $0$ or $\infty$ in the extended positive real axis.

In terms of matrices multiplication, we have 
\[ M_0M_1M_2\cdots M_n \begin{pmatrix} 0 \\ 1 \end{pmatrix} 
\sim  
M_0M_1M_2\cdots M_{n-1} \begin{pmatrix} 1 \\ 0 \end{pmatrix} \]
in ${\cal C}_2$. 
Define $p_{-1}=1$, $q_{-1}=0$, $p_0=a_0$, $q_0=1$ and 
\[ \begin{pmatrix}
    p_n & b_np_{n-1} \\ 
    q_n & b_nq_{n-1}
   \end{pmatrix} \stackrel{{\rm def}}{=} M_0M_1M_2 \cdots M_n \ 
,\ n\geq 0 \ . \]
We have the following second order recursive formulas for 
$(p_n,q_n)$:
\begin{align}
\label{recur}
& p_{n+1}=a_{n+1}p_n+b_np_{n-1} \\
& q_{n+1}=a_{n+1}q_n+b_nq_{n-1} \nonumber 
\end{align}
and the determinant formula
\begin{equation}
\label{detAn}
p_nq_{n-1}-p_{n-1}q_n = (-1)^{n-1} b_0\cdots b_{n-1}
\end{equation}

We recall the series associated with a continued fraction 
\cite{wall}:
\begin{lema}
Let $(q_n)$ denote the sequence of denominators given in 
\eqref{recur} for the continued fraction $\theta(a,b)$. 
Let 
\begin{equation}
 \rho_k=-\frac{b_k q_{k-1}}{q_{k+1}} \ ,\; k 
\in \N \ .\label{assure_conv} 
\end{equation}
Then 
\[ a_0+\frac{b_0}{a_1}\left(1+\sum_{k=1}^{n-1} \rho_1\rho_2\ldots 
\rho_k \right) = \frac{p_n}{q_n} \ , n\geq 1 \ 
. \]
\end{lema}

\begin{proof} 
For $n=1$, the sum is empty, $p_1=a_1a_0+b_0$ and $q_1=a_1$; the 
equality holds. 

Consider the telescopic sum, for $n\geq 1$,
\[ \frac{p_1}{q_1}+\sum_{k=1}^{n-1} \left( \frac{p_{k+1}}{q_{k+1}} - 
\frac{p_k}{q_k}\right) = \frac{p_n}{q_n} \ . \]
From \eqref{detAn}, 
\[ \left( \frac{p_{k+1}}{q_{k+1}} - 
\frac{p_k}{q_k}\right) = (-1)^k \frac{b_0b_1\cdots 
b_k}{q_{k+1}q_k}\ .\]
Now $\frac{b_0}{a_1}\rho_1=-\frac{b_0}{q_1}\frac{b_1q_0}{q_2} 
=-\frac{b_0b_1}{q_1q_2}=\frac{p_2}{q_2}-\frac{p_1}{q_1}$. Moreover
\[ \frac{p_3}{q_3}-\frac{p_2}{q_2}= \frac{b_0b_1b_2}{q_3q_2} = 
\frac{b_0}{q_1} \left(-\frac{b_1q_0}{q_2}\right)\left(-\frac{b_2q_1}{
q_3 }
\right) = \frac{b_0}{a_1} \rho_1\rho_2 \ . \]
Multiplicative cancelling provides the argument to 
deduce the formula 
\[ \frac{p_{k+1}}{q_{k+1}}-\frac{p_k}{q_k}= \frac{b_0}{a_1} 
\rho_1\rho_2\ldots \rho_k \]
by induction, finishing the proof. 
\end{proof}

Even though $b_n>1$ may occur in \eqref{thab}, note 
that we still have $q_{n+1}\geq 2^{(n-1)/2}$.
Indeed $q_0=1$, $q_1=a_1$ and $q_2=a_2q_1+b_1q_0>1$. Finally, 
since $a_n\geq 1$ and $b_n\geq 1$ for all $n\in \N$,
\[ q_{n+1}=a_{n+1}q_n+b_nq_{n-1} \geq q_n+q_{n-1}\geq  
2^{n/2-1}+\frac{2^{(n-1)/2}}2> 2^{(n-1)/2} \ .\]

\begin{lema}
If $(a_n)$ and $(b_n)$ are sequences on a finite alphabet 
$\A\subset [\alpha,\beta] \subset [1,\infty)$, then 
the generalized continued fraction \eqref{thab} converges. 
\label{dois}
\end{lema}

\begin{proof}
It follows from \eqref{recur} that $(q_n)$, $n\geq 1$, is 
increasing, and 
\[ |\rho_k| = \left| \frac{b_kq_{k-1}}{q_{k+1}}\right| = 
\left(1+ \frac{a_{k+1}q_k}{b_kq_{k-1}}\right)^{-1}< 
(1+\alpha/\beta)^{-1} \ . \]
Thus the series with general term $\rho_1\cdots \rho_k$ is 
bounded by a convergent geometric series.   
\end{proof}

\begin{lema}
Let $(q_n)$ denote the sequence of denominators given in 
\eqref{recur} for a generalized continued fraction, with 
$(a_n)$ and $(b_n)$ sequences in a finite alphabet $\A \subset 
[\alpha,\beta]\subset [1,\infty)$. Then 
$q_n^{1/n}$ is bounded. 
\label{logq_nlim}
\end{lema}

\begin{proof}
From \eqref{recur}, 
$q_1=a_1$, and $q_{n+1}<(a_{n+1}+b_n)q_n \leq (2\beta)q_n$. Hence 
$q_n^{1/n}\leq 2\beta \sqrt[n]{a_1}\leq 2\beta^{3/2}$, where 
the last inequality is necessary only when $a_1=\beta$. 
\end{proof}

\section{Transcendence}
\label{transc}

We now specialize the study of generalized continued fractions 
for sequences $(a_n)$ and $(b_n)$ which are generated by 
primitive substitutions. These sequences provide a wealth of examples 
of {\em stammering sequences}, defined below, following \cite{acta}.

Let us introduce some notation. The set 
${\cal A}$ is called alphabet.
A word $w$ on ${\cal A}$ is a finite or infinite sequence of 
letters in ${\cal A}$. For finite $w$, $|w|$ denotes the number 
of letters 
composing $w$. Given a natural number $k$, $w^k$ is the word 
obtained by $k$ concatenated repetitions of $w$.  Given a 
rational number $r>0$, which is not an integer, $w^r$ is the word 
$w^{\lfloor r\rfloor} 
w'$, where $\lfloor r\rfloor$ denotes the integer part of $r$ 
and $w'$ is a prefix of $w$ of length $\lceil (r-\lfloor 
r\rfloor)|w|\rceil$, where $\lceil q\rceil=\lfloor q\rfloor +1$ is 
the upper integer part of $q$. 

Note that if $(a_n)$ and $(b_n)$ are sequences on ${\cal A}$, then 
$(a_n,b_n)$ is a sequence in ${\cal A}\times {\cal A}$, which is also 
an alphabet. A sequence ${\bf a}=(a_n)$ has the {\em stammering 
property} if it is not a periodic sequence and, given $r>1$, there 
exists a sequence of finite words $(w_n)_{n\in \N}$, such that
\begin{itemize}
\item[a)] for every $n\in \N$, $w_n^r$ is a prefix of ${\bf a}$;
\item[b)] $(|w_n|)$ is increasing. 
\end{itemize}
We say, more briefly, that $(a_n)$ is a stammering 
sequence with exponent $r$. It is clear that if $(a_n)$ 
and $(b_n)$ are both stammering with exponents $r$ and $s$ 
respectively, then $(a_n,b_n)$ is also stammering with 
exponent $\min\{r,s\}$. 

\begin{lema}
If $u$ is a substitution sequence on a finite alphabet $\A$, 
then $u$ is stammering.
\label{fact1}
\end{lema}

\begin{proof}
Denote the substitution map by $\xi:\A\to \A^+$.
Since $\A$ is a finite set, there is a $k\geq 1$ and $\alpha 
\in \A$ such that $\alpha$ is a prefix of $\xi^k(\alpha)$. 
$u=\lim_{n\to \infty} \xi^{kn}(\alpha)$. Moreover, there is a 
least finite $j$ such that $\alpha$ occurs a second time in 
$\xi^{jk}(\alpha)$. Therefore $u$ is stammering with $w\geq 
1+\frac{1}{|\xi^{jk}(\alpha)|-1}$. 
\end{proof}

\begin{proof}[Proof of Theorem \ref{teo1}]
From Lemma \ref{fact1} $(a_n)$ and $(b_n)$ stammering 
sequences with exponent $r>1$. Hence 
$\theta(a,b)$ has infintely many {\em good} 
quadratic approximations. Let $(w_n)\in ({\cal A}\times {\cal 
A})^*$ be a sequence of words of increasing length 
characterizing $(a_n,b_n)$ as a stammering sequence 
with exponent $r>1$. Consider $\psi_k(a,b)$ given 
by
\[ \psi_k(a,b)= 
c_0+\frac{d_0}{c_1+\frac{d_1}{c_2+\frac{d_2}{c_3+\frac{d_3}{\ddots}}}}
 \ , 
\]
where $c_j=a_j$, $d_j=b_j$, for $0\leq j<k$, and 
$c_{j}=c_{j\pmod{k}}$ 
and $d_{j}=d_{j\pmod{k}}$, for $j\geq k$. $\psi_k$ is a root of
the quadratic equation
\[ q_{k-1}x^2+(q_k-p_{k-1})x-p_k=0 \ ,  \]
which might not be in lowest terms. 

Arguing as in Theorem 1 from \cite{acta}, we choose $k$ from the 
subsequence of natural numbers given by $|w_n^r|$. Lemma 
\ref{logq_nlim} allows us to conclude that the generalized continued 
fraction  $\theta(a,b)$ is transcendental if both $(a_n)$ and $(b_n)$ 
are stammering sequences with exponent $r>1$. 
\end{proof}

\section{Quest on Liouville numbers}
\label{liou}

We address the question of Mahler's classification of the 
numbers for type ${\mathcal S}_3$ numbers. The 
statement of Baker's Theorem we quote use a measure of 
transcendence introduced by Koksma, which is equivalent to 
Mahler's, and we explain briefly, following 
\cite{plms}, Section 2. 

Let $d\geq 1$ and $\xi$ a real number. Denote by $P(X)$ an 
arbitrary polynomial with integer coefficients, and 
$H(P)=\max_{0\leq k\leq j}  \{|a_k|\ :\ 
P(X)=a_0+a_1X+\cdots+a_jX^j\}$ is the height of the polynomial 
$P$.   Let $w_d(\xi)$ be the supremum of the real numbers such 
that the inequality
\[ 0< |P(\xi)|\leq H(P)^{-w} \]
is true for infinitely many polynomials $P(X)$ with integer 
coefficients and degree at most $d$. Koksma introduced 
$w_d^*(\xi)$ as the supremum of the real numbers $w^*$ such that
\[ 0< |\xi-\alpha| \leq H(\alpha)^{-w^*-1} \]
are true for infinitely many algebraic numbers $\alpha$ of 
degree at most $d$, where $H(\alpha)$ is the 
height of the minimal polynomial with integer 
coefficients which vanishes at $\alpha$. 

Let $w(\xi) = \lim_{d\to \infty} \frac{w_d(\xi)}{d}$, then 
$\xi$ is called 
\begin{itemize}
\item an $A$-number if $w(\xi)=0$;
\item an $S$-number if $0<w(\xi)<\infty$;
\item a $T$-number if $w(\xi)=\infty$, but $w_d(\xi)<\infty$ 
for every integer $d\geq 1$;
\item an $U$-number if $w(\xi)=\infty$ and $w_d(\xi)=\infty$ 
for some $d\geq 1$.
\end{itemize}

It was shown by Koksma that $w^*_d$ and $w^*$ provide the same 
classification of numbers. Liouville numbers are precisely 
those for which $w_1(\xi)=\infty$, they are $U$-numbers of type 
1.

\begin{teob}[Baker]
\label{baker}
Let $\xi$ be a real number and $\epsilon >0$. Assume there is 
an infinite sequence of irreducible rational numbers 
$(p_n/q_n)_{n\in \N}$, $(p_n,q_n)=1$, ordered such that $2\leq 
q_1< q_2\leq \cdots $ satisfying
\[ \left| \xi - \frac{p_n}{q_n} \right| < \frac 
1{q_n^{2+\epsilon} } \ . \]
Additionally, suppose that 
\[ \limsup_{n\to \infty} \frac{\log q_{n+1}}{\log q_n}<\infty \ 
, \]
then there is a real number $c$, depending only on $\xi$ 
and 
$\epsilon$ such that
\[ w_d^* (\xi) \leq \exp\exp ( cd^2) \ . \]
for every $d\in \N$. Consequently, $\xi$ is either an 
$S$-number or a $T$-number. 
\end{teob}

\begin{proof}[Proof of Theorem \ref{teo2}]
We note that the hypothesis of irreducibility is lacking for 
type $\S_3$ numbers. Let us write $d_n=(p_n,q_n)$. By 
eq. \eqref{detAn}, $d_n=b_0\ldots b_{n-1}$. Recall that, for 
primitive substitutions in ${\cal A}=\{\alpha,\beta\}\subset 
\N$, there is a frequency $\nu$, which is 
uniform in the sequence $({\mathbf b})$ \cite{queffe}, for 
which 
$b_k=\beta$. Thus, $d_n \approx 
\beta^\nu \alpha^{1-\nu}$ for large $n$. If, for every $n\in 
\N$, there is a number $\theta$ such that $d_n < 
\left(\frac{q_n}{d_n}\right)^\theta$, then
\[ 0< \left| \xi -\frac{p_n/d_n}{q_n/d_n}\right| < \frac 
1{q_n^{2+\epsilon}}= \frac 1{d_n^{2+\epsilon} 
(q_n/d_n)^{2+\epsilon}} \ .\]
In this case, from the estimates of $q_n$ and $d_n$, the limit 
\[ \limsup_{n\to \infty} 
\frac{\log(q_{n+1}/d_{n+1})}{\log(q_n/d_n)} <\infty \ .\]
We would conclude from 
Theorem \ref{baker} that type $\S_3$ contains  either 
$S$-numbers or $T$-numbers and no Liouville numbers.

From the analysis of Section \ref{conv}, keeping its notations,
\[ (-1)^n \frac{d_n}{q_nq_{n-1}} = \frac{b_0}{a_1} \rho_1 
\ldots \rho_{n-1} \ . \]
Now 
$|\rho_k|=\left(1+\frac{a_{k+1}q_k}{b_kq_{k-1}}\right)^{-1}$, 
and since $q_k\leq 2\beta q_{k-1}$, we conclude that 
\[ |\rho_k| > \left( 1+ \frac{2\beta^2}{\alpha}\right)^{-1} \ . 
\]
Therefore, recalling that $q_{n-1}\geq 2^{(n-3)/2}$
\[ \frac{d_n}{q_n} > q_{n-1} (1+2\beta^2/\alpha)^{-n+1} 
\frac{b_0}{a_1} \quad \Rightarrow\quad \frac{q_n}{d_n}<
(1+2\beta^2/\alpha)^{n-1} 2^{-(n-3)/2}\frac{\beta}{\alpha} \ . 
\]

Therefore, we want to determine the 
existence of a solution for $\theta$ for the inequality
\[  d_n  < \left(\frac{q_n}{d_n}\right)^\theta \]
considering that $d_n \approx \beta^{\nu n} \alpha^{(1-\nu)n}$
we obtain the inequality 
\[ \beta^{\nu n} 
\alpha^{(1-\nu)n} 
< (1+2\beta^2/\alpha)^{\theta(n-1)}2^{-\theta(n-3)/2} 
\frac{\beta}{\alpha} \ .\]
For large $n$, it is sufficient to solve 
\[ \beta^{\nu}\alpha^{1-\nu} < 
\frac 
1{2^{\theta/2}}\left(1+\frac{2\beta^2}{\alpha}\right)^\theta  \ 
,\]
which clearly has the solution $\theta=1$, since 
$\alpha<\beta$ and $0<\nu<1$, implying $\beta^\nu 
\alpha^{1-\nu} < \beta$. We conclude that type 
$\S_3$ 
consists 
only of $S$-numbers or $T$-numbers. 
\end{proof}

\section{Example: partial quotients of a corresponding regular 
continued fraction}
\label{exam1}

We now examine one specific example: a generalized continued fraction 
associated with the period doubling sequence. The 
period doubling sequence, which we denote by $\omega$, 
is the fixed point of the substitution 
$\xi(\alpha)=\alpha\beta$ and 
$\xi(\beta)=\alpha\alpha$ on the two lettered alphabet 
$\{\alpha,\beta\}$. 
It is 
also the limit of a sequence of {\em foldings}, and called a 
{\em folded 
sequence} \cite{allsha}.

We make some observations and one question about 
the partial 
quotients of the 
corresponding regular continued fraction representing the  
real number that the generalized continued 
fraction given by \eqref{thab} when both sequences 
$({\mathbf a})$ and $({\mathbf b})$ are given by the period 
doubling sequence: $a_n=b_n=\omega_n$.

We choose to view the period doubling sequence as the 
limit of folding operations. The algebra of matrices with fixed 
determinant will play a role. A 
folding is a mapping 
\begin{align*}
{\cal F}_p: &{\cal A}^*\to {\cal A}^*\\
& w\mapsto wp\tilde{w}  
\end{align*}
where $\tilde{w}$ equals the word $w$ reversed: if $w=a_1\ldots 
a_n$, $a_i\in {\cal A}$, then $\tilde{w}=a_n\ldots a_1$, and 
$p\in {\cal A}^*$. 

It is clear that
\[ \omega= \lim_{n\to \infty} ({\cal F}_a\circ {\cal F}_b)^n 
(a) \ ,\]
see also \cite{allsha},
where the limit is understood in the product topology 
(of the discrete topology) in 
${\cal A}^\N\cup {\cal A}^*$.

Let $\theta$ denote the number whose generalized continued 
fraction is obtained from the substitution of the letters 
$\alpha$ and $\beta$ by 
\[ A=\begin{pmatrix} 1 & 1 \\ 1 & 0 \end{pmatrix} \ , \quad B= 
\begin{pmatrix} 3 & 3 \\
 1 & 0 
\end{pmatrix} \]
respectively. It corresponds to the choice $\{1,3\}$ for the 
alphabet where the sequences $({\mathbf a})$ and 
$({\mathbf b})$ take values.

Now we use Raney transducers \cite{raney} to describe 
the computation of some partial quotients of the regular 
continued fraction converging to $\theta$. A 
transducer ${\mathscr T}=(Q,\Sigma,\delta,\lambda)$, or 
two-tape machine, is defined by a set of states $Q$, an 
alphabet $\Sigma$, a transition function $\delta: Q\times 
\sigma \to Q$, and an output function $\lambda: Q\times \sigma 
\to \Sigma^*$, where $\sigma\subset \Sigma$ (a more general 
definition is possible \cite{allsha}, but this is sufficient 
for our purposes).

The states of Raney's tranducers are column and 
row (or doubly) balanced matrices over the non-negative 
integers 
with a fixed determinant. A matrix $\begin{pmatrix} a & b \\ c 
& 
d\end{pmatrix}$ is column balanced if $(a-b)(c-d)<0$ 
\cite{raney}. 

Figure 1 shows the Raney transducer for determinant 3 doubly 
balanced matrices. In the text, we use the abbreviations: 
$\beta_1=\begin{pmatrix} 3 & 0 \\ 0 & 1 \end{pmatrix}$, 
$\beta_2=\begin{pmatrix} 1 & 0 \\ 0 & 3 \end{pmatrix}$ and 
$\beta_3=\begin{pmatrix} 2 & 1 \\ 1 & 2 \end{pmatrix}$. Then 
$Q=\{\beta_1,\beta_2,\beta_3\}$, $\Sigma=\{L,R\}$, where 
\[ R=\begin{pmatrix}
      1 & 1 \\ 0 & 1 
     \end{pmatrix} \ , \quad L=\begin{pmatrix} 1 & 0 \\ 1 & 1 
\end{pmatrix} \ ,\]
the 
transition function $\delta$ and the output function are 
indicated in the graph. 

For instance, if $RL^2R$ is the input word on state $\beta_2$, 
the output word is $L^2R^4$ and the final state is $\beta_1$. 
Any infinite word in $\Sigma^\N$ can be read by ${\mathscr 
T}$, but not every finite word can be fully read by ${\mathscr 
T}$, for 
instance, $L^{11}$ in state $\beta_1$ will produce $L^3$, but 
$L^2$ will stay in the reading queue in state $\beta_1$. 
Algebraically, these two examples are written as
\begin{align*}
\beta_2 RL^2R & =  L \beta_3 LR = L LR \beta_1R = L^2 RR^3 
\beta_1 = L^2R^4 \beta_1 \\
\beta_1 L^{11} &=  L^3 \beta_1 L^2 
\end{align*}

As explained in \cite{vdP}, Theorem 1 in \cite{list} or even  
Theorem 5.1 in \cite{raney}, one may use the transducer 
${\mathscr T}$ to commute the matrices $A$ and $B$ to get an 
approximation to the continued fraction of $\theta$.

\begin{figure}
\centering 
\begin{tikzpicture}[->, >=stealth', auto, semithick, node 
distance=3.5cm]
\tikzstyle{every state}=[fill=cyan,draw=none,thick,text=white,
scale=1.4]
\node[state] (Am) {$\displaystyle\left(\begin{smallmatrix} 3 & 
0 
\\ 0 & 1 \end{smallmatrix}\right)$};
\node[state] (B) [above right of=Am] 
{$\displaystyle\left(\begin{smallmatrix} 2 & 1 \\ 1 
& 2 \end{smallmatrix}\right)$};
\node[state] (Ad) [below right of=B] 
{$\displaystyle\left(\begin{smallmatrix} 1 & 0 \\ 0 
& 3 \end{smallmatrix}\right)$};
\path
(B) edge [bend right=12] node[above,yshift=0.2cm,xshift=-0.2cm] 
{$L/LR$} (Am)
edge  [bend left=12] node [xshift=0.1cm,yshift=-.2cm] 
{$R/RL$} (Ad)
(Am) 
edge [bend right=20] node [below,yshift=.1cm] {$L^2R/RL^2$} 
(Ad)
edge [bend right=20] node [below,yshift=0.2cm,xshift=0.6cm] 
{$LR/R$} (B)
edge [out=245,in=305,looseness=8] node [xshift=0.4cm]  
{$L^3/L$} (Am)
edge [out=175,in=235,looseness=8] node [xshift=-0.1cm]  
{$R/R^3$} (Am)
(Ad) edge [out=305,in=5,looseness=8] node {$L/L^3$} (Ad)
edge [out=235,in=295,looseness=8] node  [xshift=-0.4cm] 
{$R^3/R$}
(Ad)
edge [bend left=20] node [below,xshift=-0.3cm] {$RL/L$} (B)
edge [bend right=12] node [below] {$R^2L/LR^2$} (Am)
;
\end{tikzpicture}
\caption{Transducer $\mathscr T$.}
\label{trans1}
\end{figure}
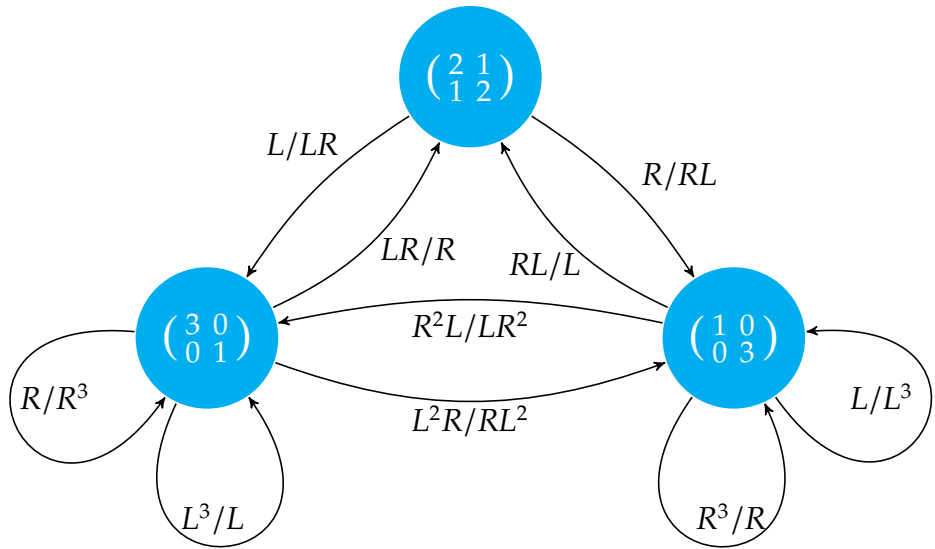

\def\A{\mathscr A}
\def\B{\mathscr B}
Introducing the matrix $J=\begin{pmatrix} 0 & 1 \\ 1 & 0 
\end{pmatrix}$, we note the following relations: 
$A=RJ$, $B=\beta_1 RJ$, $\beta_2 J=J\beta_1$.

The homomorfism between ${\cal A}^*$ and the semigroup of 
matrices generated by $\{A,B\}$, as Moebius transforms, is 
the basis for discovering some curious properties of the 
regular continued 
fraction of $\theta$. 

The sequence of matrices  
\[ ABA,\ ABAAABA,\ ABAAABABABAAABA \ ,\]
which corresponds to ${\cal F}_b(a)$, ${\cal F}_a\circ{\cal F}_b(a)$, 
${\cal F}_b\circ {\cal F}_a\circ {\cal F}_b(a)$, yields the 
beginning of the regular continued fraction expansion of 
$\theta$. 

A step by step calculation shows the basic features in the use 
of ${\mathscr T}$:
\begin{align*}
ABAAABA &= RJ(\beta_1 RJ) RJRJRJ(\beta_1 RJ)RJ \\
&= R \beta_2 (JRJ)R(JRJ)R\beta_2 (JRJ)RJ \\
&= R \beta_2 LRLR\beta_2 LRJ \\
& = R L^3 \beta_2 RLR L^3 \beta_2 RJ \\
& =RL^3 L \beta_3 RL^3 \beta_2 RJ \\
& = RL^4 RL \beta_2 L^3 \beta_2 RJ \\
&= RL^4 RL L^9 \beta_2^2 RJ = RL^4 RL^{10} \beta_2^2 RJ \ .
\end{align*}
This means that the continued fraction of $\theta$ begins as 
$[1;4,1,k,\cdots]$, with $k\geq 10$.

Writing $T=ABAAABA$, upon the next folding
\begin{align*}
TBT &= RL^4RL^{10} \beta_2^2 RJ (\beta_1RJ) RL^4RL^{10} 
\beta_2^2 RJ \\
&= RL^4 RL^{10} \beta_2^2 R\beta_2 LRL^4RL^{10} \beta_2^2 RJ \\
&= RL^4RL^{10} \beta_2^2RL^3L\beta_3 L^3
RL^{10} \beta_2^2 RJ \\
& = RL^4R L^{10} \beta_2^2 RL^4LR \beta_1 L^2R 
L^{10} \beta_2^2 RJ \\
&= RL^4R L^{10} \beta_2^2 RL^5 RRL^2 \beta_2 L^{10} 
\beta_2^2 RJ \\
&=RL^4RL^{10} \beta_2 L \beta_3L^4 R^2 L^{32} \beta_2^3 RJ \\
&= RL^4RL^{10} \beta_2 L \beta_3 L^4 R^2 L^{32}\beta_2^3 RJ 
\\
&= RL^4 R L^{13} \beta_2 LR \beta_1 L^3 R^2L^{32} \beta_2^3 RJ 
\\
& = RL^4 RL^{16} \beta_2 R L \beta_1 R^2 L^{32} \beta_2^3 RJ \\
&= RL^4 RL^{17} \beta_3 R^6L^{10} \beta_1 L^2 \beta_2^3 RJ \\
&= RL^4 RL^{17} RL \beta_2 R^5L^{10} \beta_1 L^2\beta_2^3 RJ \\
&= RL^4 RL^{17} RL R LR^2 \beta_1 L^9 \beta_1 L^2 \beta_2^3 RJ 
\\
&= RL^4 RL^{17} RLRLR^2L^3 \beta_1^2 L^2 \beta_2^3 RJ 
\end{align*}
Now we have the knowledge that the beginning of the regular 
continued fraction expansion of 
$\theta=[1;4,1,17,1,1,1,1,2,a,\cdots]$, with $a\geq 3$. 

Instructions for the transducer are stuck on the right of this 
factorization to be read for the next folding. Note that only 
states $\beta_1$ and $\beta_2$ will remain, since a transition 
from $\beta_3$ is always possible given any finite word in 
$\Sigma^*$. 

An inductive prediction as to whether the high power in the 
beginning, $L^{17}$, will consistently increase upon 
(sufficient) repetitions of foldings 
is out of reach. This observation poses the question: are the 
partial quotients of regular continued fraction of $\theta$ 
bounded? Similar calculations have been done with a simpler 
choice of the alphabet, that is, $\{1,2\}$, where the 
transducer has only two states (doubly balanced matrices with 
determinant 2).


\begin{thebibliography}{99}


\bibitem{acta} B. Adamczewski, Y. Bugeaud, On the complexity of 
algebraic numbers, II. Continued fractions. Acta Math. {\bf 
195}, 1--20 (2005) 

\bibitem{plms} B. Adamczewski and Y. Bugeaud, Mesures de 
transcendance et aspects quantitatifs de la méthode de 
Thue-Siegel-Roth-Schmidt, Proc. London Math. Soc. {\bf 101} 
1--26 (2010)

\bibitem{abd} B. Adamczewski, Y. Bugeaud, L. Davison, Continued 
Fractions and Transcendental Numbers. Ann. Inst. Fourier {\bf 
56}, n. 7, 2093--2113 (2006) 

\bibitem{aldavquef} J.-P. Allouche, J.L. Davison, M. Queffélec, 
L. Q. Zamboni, Transcendence of Sturmian or Morphic Continued 
Fractions. J. Number Th. {\bf 91}, 39--66 (2001)

\bibitem{allsha} J.-P. Allouche, J. Shallit, {\bf Automatic 
Sequences: Theory, Applications, Generalizations}. Cambridge 
University Press, Cambridge, (2003)

\bibitem{bruno}
A. D. Bruno, New Generalization of Continued Fraction, I. 
Functiones et Approximatio, {\bf 43}, 55--104, (2010) 

\bibitem{garrity} 
T. Garrity, On periodic sequences for algebraic numbers. J. 
Number Th. {\bf 88}, 83–103 (2001)

\bibitem{keane}
M. Keane, Irrational Rotations and Quasi-Ergodic 
Measures. Pub. Sém. Math. Inf. Rennes, {\bf 1}, 17--26 (1971)

\bibitem{list} P. Liardet and P. Stambul, Algebraic Computations 
with Continued Fractions. J. Numb. Th. {\bf 73}, 92--121 (1998) 

\bibitem{meno} A. Messaoudi, A. Nogueira, F. Schweiger, Ergodic 
Properties of Triangle Partitions. Monatsh. Math. DOI 
10.1007/s00605-008-0065-z (2008)


\bibitem{queffe} M. Queffelec, {\bf Substitution Dynamical 
Systems -- Spectral Analysis}. 2nd ed., Lect. N. Math. {\bf 
1294}, Springer, Berlim (2010) 


\bibitem{raney} G. N. Raney, On Continued Fractions and Finite 
Automata. Math. Ann. {\bf 206}, 265--283 (1973)

\bibitem{vdP} A. J. van der Poorten, An Introduction to 
Continued Fractions. In {\bf Diophantine Analysis}, London 
Math. Soc. Lecture Notes {\bf 109}, Cambridge Univ. Press, 
(1986) 

\bibitem{schwe} 
F. Schweiger, {\bf Multidimensional Continued Fractions}. 
Oxford University Press, Oxford (2000)

\bibitem{wall} H. S. Wall, {\bf Analytic Theory of Continued 
Fractions}. Van Nostrand, New York, (1948)


\end{thebibliography}
\end{document}